\documentclass{amsart}

\usepackage{amssymb}
\usepackage{mathtools}
\mathtoolsset{showonlyrefs}
\usepackage[stretch=10,shrink=10]{microtype}
\usepackage[shortlabels]{enumitem}
\usepackage{tikz}
\usetikzlibrary{arrows.meta}
\usepackage{comment}
\usepackage{subfigure}
\usepackage{pgfplots}
\usepackage{nccmath}

\newcommand{\HOX}[1]{\marginpar{\scriptsize #1}}
 \newcommand{\f}{\frac}

\renewcommand{\phi}{\varphi}

\newcommand{\E}{\mathbf E}
\renewcommand{\P}{\mathbf P}
\renewcommand{\Pr}{\P}
\newcommand{\TT}{\mathbb T}

\newcommand{\R}{\mathbb{R}}

\newcommand{\SA}{\texttt{SelfActivate}}
\newcommand{\NV}{\texttt{NoVisit}}

\DeclareMathOperator{\Poi}{Poi}

\DeclareMathOperator{\SFM}{SFM}
\DeclareMathOperator{\FM}{FM}

\DeclareMathOperator{\Comp}{Comp}
\DeclareMathOperator{\len}{len}

\newcommand{\argmax}{\operatorname*{arg\,max}}
\newcommand{\abs}[1]{\lvert #1 \rvert}

\usepackage{hyperref}
\hypersetup{colorlinks,linkcolor={blue},citecolor={olive},urlcolor={brown}}
\usepackage[all]{hypcap}
\usepackage{theoremref}

\newtheorem{theorem}{Theorem}
\newtheorem{lemma}[theorem]{Lemma}
\newtheorem{claim}[theorem]{Claim}
\newtheorem*{lemma*}{Lemma}

\newtheorem*{conjecture}{Conjecture}
\theoremstyle{definition}

\newtheorem*{remark*}{Remark}

\title{Improved critical drift estimates for the frog model on trees}

\author{Poly Mathews Jr.} \email{Matthew.Junge@baruch.cuny.edu; sit218@lehigh.edu}

\thanks{This project was initiated during the 2023 Polymath Jr.\ Program. The authors are: 
Emma    Bailey,
Yanni	Bills,
Feng	Cheng,
Osanda	Chinthila,
Aryan	Dugar,
Ali	Eser,
Eric	Han,
Chirag Kar,
Matthew	Junge,
Viet	Le,
Jiaqi	Liu,
Zoe	McDonald,
Si	Tang,
Michael	Thomas,
Scott	Wynn, and
Eric	Yu.
The research was partially supported by NSF Grants CAREER-2238272, DMS-2115936, DMS-2218374, and DMS-2137614.  
}

\begin{document}
\maketitle

\begin{abstract}
 Place an active particle at the root of the infinite $d$-ary tree and dormant particles at each non-root site. Active particles move towards the root with probability $p$ and otherwise move to a uniformly sampled child vertex. When an active particle moves to a site containing dormant particles, all the particles at the site become active. The critical drift $p_d$ is the infimum over all $p$ for which infinitely many particles visit the root almost surely. We give improved bounds on $\sup_{d \geq m} p_d$ and prove monotonicity of critical values associated to a self-similar variant.
\end{abstract}




\section{Introduction}

The \emph{frog model with drift} on the infinite rooted $d$-ary tree $\TT_d$ with initial probability measure $\nu$ supported on $[0,\infty)$, denoted $\FM(d,p,\nu)$, is defined as follows. Initially one active particle is at the root $\varnothing\in \mathbb T_d$ and the other vertices have independent and identically $\nu$-distributed (i.i.d.)\ many dormant particles. Active particles perform independent $p$-biased random walk, i.e., moving toward the root with probability $p$ and otherwise moving to a uniformly random child vertex. When an active particle moves to an unvisited site, any dormant particles there become active. These dynamics capture aspects of spatial processes with activation such as combustion, rumor spread, and infection \cite{ramirez2004asymptotic}. Due to the chaotic nature in which the model propagates, researchers referred to particles as awake and sleeping frogs. The zoomorphism has persisted, so we will often use that terminology. 

A \emph{root visit} occurs each time that an awake frog moves to $\varnothing$. Let $V = V_{\FM(d,p,\nu)}$ be the total number of root visits. Call the process \emph{recurrent} if $V = \infty$ almost surely. It is proven in \cite{beckman2019frog} that $V$ satisfies a $0$-$1$ law. Accordingly, we call the process \emph{transient} in the other case that $V < \infty$ almost surely. 

The \emph{critical drift}
\[
p_d(\nu) \coloneqq \inf\{p\colon \FM(d,p,\nu) \text{ is recurrent}\}
\]
is the minimal drift below which the process is transient.
It is known that $p_d(\nu)$ is sensitive to more than just the mean of $\nu$ \cite{johnson2018stochastic, johnson2019sensitivity}. Most interest has been in one-per-site and Poisson-distributed initial configurations. We will write $\FM(d,p,1)$ to denote the one-per-site case $\nu(1)=1$ and $\FM(d,p,\Poi(\mu))$ to denote the case $\nu$ has a Poisson distribution with mean $\mu$. It is useful to standardize the density of sleeping frogs. Unless otherwise indicated we write  \[
p_d \coloneqq \inf\{p\colon \FM(d,p,1) \text{ is recurrent}\}
\]
for the critical drift in the one-per-site frog model.

The frog model was first studied on $\mathbb Z^d$. The frog model on $\mathbb Z^d$ is related to branching random walks on polynomially growing graphs. Telcs and Wormald \cite{telcs1999branching} proved that the one-per-site version is recurrent in all dimensions. Telcs and Wormald \cite{telcs1999branching} and Popov \cite{popov2001frogs} proved that it is recurrent for any $\nu$ with $\nu(0) \neq 1$. Alves, Machado, and Popov as well as Ram\'irez and Sidoravicius showed that the set of visited sites has a limiting shape for the one-per-site model \cite{alves2002shape, ramirez2004asymptotic}.  There has also been interest in the frog model on $\mathbb Z^d$ in which particles have a drift in one coordinate direction \cite{gantert2009recurrence, dobler2018reccurence}. 

On $\mathbb T_d$, a simple random walk corresponds to $p=\tilde p\coloneqq 1/(d+1)$. Determining the transience/recurrence behavior of $\FM(d,\tilde p,1)$ was open for over a decade \cite{popov2003frogs}. The question was partially answered in \cite{hoffman2017recurrence} by Hoffman, Johnson, and Junge who proved that $\FM(d,\tilde p,1)$ is recurrent for $d=2$ and transient for $d\geq 5$. Simulations suggest that the process is recurrent when $d=3$ and transient when $d=4$. Later, the same authors in \cite{hoffman2016transience, johnson2016critical} proved that there is a critical $\mu_c(d)=\Theta(d)$ above which $\FM(d,\tilde p,\Poi(\mu))$ is recurrent and below which it is transient.

Wanting to isolate the role of the drift from the tree structure in the phase transition, Beckman, Frank, Jiang, Junge, and Tang introduced $\FM(d,p,1)$ and its critical value $p_d$ \cite{beckman2019frog}. It is easy to see that the stochastically larger process with all particles initially active is transient whenever $p < 1/(d+1)$. Moreover, the process with no activation is recurrent for $p\geq 1/2$. So, $p_d \in [1/(d+1),1/2].$ An intriguing aspect of $\FM(d,p,1)$ is that several intuitive monotonicity statements have evaded proof. We say that $Y$ \emph{stochastically dominates} $X$ if there is a coupling such that $X \leq Y$ almost surely. This is denoted by $X \preceq Y$. 

 \begin{conjecture}[\cite{beckman2018asymptotic, guo2022minimal, bailey2023critical}]
 \quad 
     \begin{enumerate}[label = (\roman*)]
        \item If $d\leq d'$ and $p \leq p'$, then $V_{\FM(d,p,\nu)}\preceq V_{\FM(d',p',\nu)}$.\footnote{Conjecture (i) is known in the special case that $d'=kd$ and $p'=p$ \cite[Proposition 1.2]{beckman2019frog}.} 
        \item $p_{d+1} < p_d$.
        \item $\lim_{d \to \infty} p_d = (2-\sqrt 2)/4 \coloneqq q^\ast$ ($\approx 0.1464$) the critical drift for a branching random walk that doubles only when moving away from the root.
     \end{enumerate}
 \end{conjecture}

 The main result of \cite{hoffman2016transience} can be restated as $p_2 = 1/3$. However, the lack of monotonicity in Conjecture (i) makes it unclear if \[S_m \coloneqq \sup_{d \geq m} p_d\] is bounded from above by $1/3$. In testament to this uncertainty, the first bound proven on $S_m$ was that $S_3 \leq 0.4155$ \cite[Theorem 1.1]{beckman2019frog}. Guo, Tang, and Wei later established the bound $S_3 \leq 1/3$ which implies the ``sharp'' statement $S_2 = 1/3$  \cite{guo2022minimal}. This was further improved by Bailey, Junge, and Liu to $S_3 \leq 5/17$ \cite{bailey2023critical} which implies that $p_d < p_2$ for $d \geq 3$.  Bailey, Junge, and Liu also proved that $S_4 \leq 27/100$ and outlined a computer-assisted method for obtaining better bounds for larger $m$.

 \subsection{Results}
 
 Our results provide further insight into Conjecture (i), (ii), and (iii). 
  The first uses a computer-assisted proof to carry out the proposed method from \cite{bailey2023critical} for bounding $S_m$. The general idea is to bound the critical drift for a non-backtracking variant that has stochastically fewer root visits (see \thref{thm:m1}).

 \begin{theorem} \thlabel{thm:Sm}
     $S_m$ satisfies the bounds in the table below.
     \begin{center}
     \small 
     \begin{tabular}{c|cccccccccccc}
         $m$ & $2$ & $3$ & $4$ & $5$ & $6$ & $7$ & $8$ & $9$ & $10$ & $11$ & $12$ & $13$ \\
         \hline
         \\
         \vspace{-.6 cm}
         \\
         $S_m \leq$ & $\f{55}{159}$ & $\f{42}{145}$ & $\f{40}{153}$ & $\f{23}{94}$ & $\f{46}{197}$ & $\f{23}{102}$ & $\f{38}{173}$ & $\f{20}{93}$ & $\f{15}{71}$ & $\f{5}{24}$ & $\f{7}{34}$ & $\f{11}{54}$
         \\
         \vspace{-.3 cm}
         \\
          $\approx$ & $.346$ & $.290$ & $.261$ & $.245$ & $.234$ & $.225$ & $.220$ & $.215$ & $.211$ & $.208$ & $.206$ & $.204$
    \end{tabular}
    \end{center}
 \end{theorem}


\begin{figure}
\centering
\includegraphics[width = .9\textwidth]{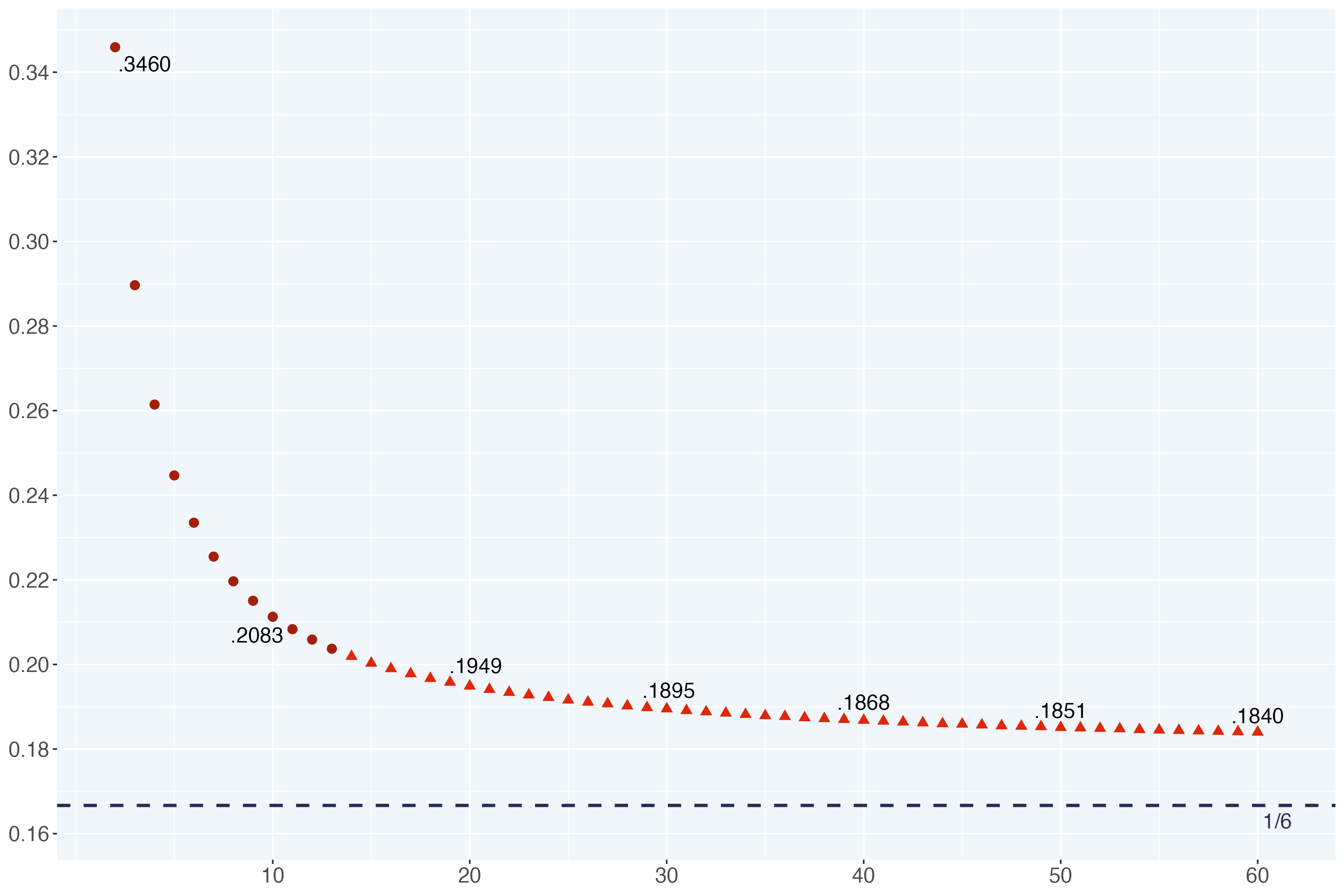}
\caption{Bounds on $S_m$ for $2 \leq m \leq 60$. The horizontal axis is the value of $m$ and the vertical axis is the bound we get for $S_m$. Circles are rigorous bounds from \thref{thm:Sm} and triangles are numerically approximated bounds. The dashed line at $1/6$ is the known limit of the sequence of upper bounds (not the limit of $S_m$).}\label{fig:up-bounds}
\end{figure}

Beyond $m=13$ we encounter runtime issues (even getting to $m=13$ requires the efficiency-boosting inductive scheme described in \thref{lem:inductive}). However, we are able to use the method to make non-rigorous approximations of bounds for larger $m$. See Figure~\ref{fig:up-bounds}.
It is proven in \cite{bailey2023critical} that the upper bounds on $S_m$ will converge to $1/6$ as $m \to \infty$ (although $S_m$ ought to converge to $q^*$ from Conjecture (iii)). The convergence in Figure~\ref{fig:up-bounds} appears slow. This leads to a refinement to Conjecture (iii). We find it plausible that $p_d - q^* = \Theta(d^{-1/2})$. This might occur because for large $d$ the frog model behaves like the branching random walk that doubles when moving away from the root until vertices start getting revisited. The birthday paradox tells us that repeated visits to child vertices of $\varnothing$ occur after $O(\sqrt d)$ visits to $\varnothing$. The branching random walk visits $\varnothing$ at a constant (linear in time) rate, as does the frog model  \cite{hoffman2019infection} in certain regimes. This suggests that $d^{-1/2}$ may play a role in the point at which the frog model begins to lag behind the branching random walk.


The rest of our results are for an important variant of $\FM(d,p,\nu)$ called the \emph{self-similar frog model}, which we denote as $\SFM(d,p,\nu)$. This process, introduced in \cite{hoffman2017recurrence}, is the {only} known tool for proving recurrence of a frog model on trees. Put briefly, frogs in the self-similar frog model are restricted to the non-backtracking (loop-erased) portion of their random walk paths, and only one frog is allowed to move away from the root to each subtree. This results in a stochastically smaller model in terms of the root visits that is more amenable to anaylsis.  



In $\SFM(d,p,\nu)$ the jump distribution is different for frogs that just woke up versus for those that have already taken a step. Let
\begin{align}
	p_d^* = p_d^*(p) \coloneqq \f{ p(d-1)}{d - (d+1) p} \text{ and } \hat p = \hat p(p) \coloneqq \f{p}{1-p}. \label{eq:p*}
\end{align}
Initially, there is one active frog at the root. It moves to a uniformly sampled child vertex in the first step. Just activated frogs move towards the root with probability $p^*_d$, and otherwise away from the root to a uniformly sampled child vertex. For subsequent steps, if the previous step was towards the root, then the next step will be towards the root with probability $\hat p$. If the previous step was away from the root, all subsequent steps will be away from the root to uniformly sampled child vertices. Any particles that visit the root are killed there and no longer participate in the process. The last modification is that particles moving away from the root are killed upon visiting a vertex that has already been visited. If multiple active particles attempt to move away from the root to the same unvisited vertex, then one is chosen to continue its path and the others are killed. 

Let $V_{\SFM(d,p,\nu)}$ denote the total number of root visits. The self-similar frog model is dominated by the usual frog model. Indeed, \cite{guo2022minimal} worked out the transition probabilities $p_d^*$ and $\hat p$ so that $V_{\SFM(d,p,\nu)} \preceq V_{\FM(d,p,\nu)}$. In support of Conjecture (i), we prove that $V_{\SFM(d,p,\nu)}$ is monotone in $d$. 
\begin{theorem}\thlabel{thm:m1}
    $V_{\SFM(d,p,\nu)} \preceq V_{\SFM(d+1,p,\nu)}$ for all $p \in (0,1/2), d\geq 2$ and $\nu$.
\end{theorem}

 \thref{thm:m1} is the strongest contribution of this work. It is interesting foremost because it supports Conjecture (i). Another benefit of \thref{thm:m1} is that the main results from  \cite{beckman2018asymptotic} and \cite{guo2022minimal} are immediate corollaries. For example, since \cite{hoffman2017recurrence} proved that $\SFM(2,1/3,1)$ is recurrent, \thref{thm:m1} implies that $S_3 \leq 1/3$. Also, the bounds in \cite{bailey2023critical} are strengthened after applying \thref{thm:m1}. For example, together with Theorems \ref{thm:Sm} and \ref{thm:m1}, the bound $S_4 \leq 0.27$ from \cite{bailey2023critical} is improved to $S_4 \leq 0.262$.

The main difficulty with proving \thref{thm:m1} is that $\SFM(d,p,\nu)$ and $\SFM(d+1,p,\nu)$ have different probabilities  that the first step taken by a newly activated frogs is towards the root ($p_d^* < p_{d+1}^*$). In one way this is good for $\SFM(d+1,p,\nu)$ since frogs are more likely to move towards the root. However, it is not monotonically helpful since moving away from the root sometimes comes with the benefit of waking more frogs. A fortunate inequality, that had previously gone unnoticed, is that so long as at least one vertex below a just activated site, say $v$, has been visited, the probability a particle activated at $v$ moves away from the root to a new site in $\SFM(d+1,p,\nu)$ is larger than the probability in $\SFM(d, p,\nu)$. An innovation in our coupling is to allow frogs that have jumped away from the root to keep jumping until reaching a freezing barrier. This ensures that newly awoken frogs will have at least one visited child vertex below them. Thus, there is a way to couple the two models, after sometimes eliminating frogs from $\SFM(d+1,p,\nu)$, so that there is a one-to-one correspondence between active frogs at all distances from the root.

 Our last result is more technical and involves Conjecture (ii). We define a new critical value for $\SFM(d,p,\Poi(1))$ and prove that it is strictly monotone. Informally speaking, the critical value is the smallest value of $p$ such that the only approach for proving $\SFM(d,p,\Poi(1))$ is recurrent applies. To define this formally takes some extra notation. 
 
Consider a star graph with root $\varnothing$, central vertex $\varnothing '$, and leaves $v_1, \hdots, v_d$ (see Figure \ref{fig:Aasystem}). There is a $\Poi(1)$ number of active particles at $\varnothing'$ and an independent $\Poi(\lambda)$ number of active particles at $v_1$. An independent $\Poi(\lambda)$-distributed number of dormant particles is placed at each of $v_2,\hdots, v_d$. 

 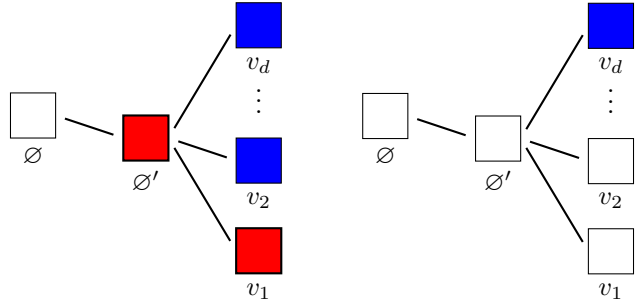
\begin{figure}
  \begin{center}
\mbox{
\subfigure
{
\begin{tikzpicture}[scale = .6]
\draw (-5,3) -- node[below] {$\varnothing$} ( -4,3) -- node (root) {} ( -4,4) -- (-5,4) -- cycle; 

\draw[thick, fill = red] (-2.5,2.5) -- node[below] {$\varnothing'$} (-1.5,2.5) -- node (root') {} (-1.5,3.5) -- (-2.5,3.5) -- node (root'2) {} (-2.5,2.5) --cycle;
\draw[fill = blue] (0,5) --node[below] {$v_d$} (1,5) -- (1,6) -- (0,6) -- node (vd) {} (0,5)--cycle;
\node at (.5,4) {$\vdots$};

\draw[fill = blue] (0,2) --node[below] {$v_2$} (1,2) -- (1,3) -- (0,3) -- node (v2) {}  (0,2)--cycle;
\draw[thick,fill = red] (0,0) --node[below] {$v_1$} (1,0) -- (1,1) -- (0,1) --  node (v1) {}  (0,0)--cycle;
\draw[thick] (root') -- (vd);
\draw[thick] (root') -- (v2);
\draw[thick] (root') -- (v1);
\draw[thick] (root'2) -- (root);

\end{tikzpicture}
}

\subfigure{

\begin{tikzpicture}[scale = .6]
\draw[thick] (root'2) -- (root);
\draw (-5,3) -- node[below] {$\varnothing$} ( -4,3) -- node (root) {} ( -4,4) -- (-5,4) -- cycle; 
\draw (-2.5,2.5) -- node[below] {$\varnothing'$} (-1.5,2.5) -- node (root') {} (-1.5,3.5) -- (-2.5,3.5) -- node (root'2) {} (-2.5,2.5) --cycle;
\draw[fill = blue] (0,5) --node[below] {$v_d$} (1,5) -- (1,6)  -- (0,6) -- node (vd) {} (0,5)--cycle;
\node at (.5,4) {$\vdots$};
\draw (0,2) --node[below] {$v_2$} (1,2) -- (1,3) -- (0,3) -- node (v2) {}  (0,2)--cycle;
\draw (0,0) --node[below] {$v_1$} (1,0) -- (1,1) -- (0,1) --  node (v1) {}  (0,0)--cycle;
\draw[thick] (root') -- (vd);
\draw[thick] (root') -- (v2);
\draw[thick] (root') -- (v1);
\node at (-6,6) {};
\end{tikzpicture}
 }
}
  \end{center}
  \caption{The process used to define $U(d,p,\lambda)$. Red sites contain particles that are initially active and blue sites contain initially dormant particles. $U(d,p,\lambda)$ is the number of vertices among $v_2,\hdots,v_d$ that are ever visited. Empty boxes on the right at $v_1,\hdots, v_d$ represent sites at which particles have been activated.}
  \label{fig:Aasystem}
\end{figure}

The active particles started at $\varnothing'$ move to $\varnothing$ independently with probability $p^*_d$ and otherwise each moves to an independently and uniformly sampled vertex from $v_1,\hdots, v_d$. Active particles at $v_i$ move to $\varnothing '$ with probability $1$, and then to either $\varnothing$ with probability $\hat p$ or otherwise to a uniformly sampled vertex among $\{v_1,\hdots, v_d\} \setminus \{v_i\}$. Whenever active particles encounter dormant particles, the dormant particles become active. When a particle moves to a leaf  or to $\varnothing$, it remains frozen there for all subsequent time steps. 
%
We define $U= U(d,p,\lambda)$ to be how many of $v_2,\hdots,v_d$ have been visited after all particles are either frozen or dormant. The random variable $U$ is important because understanding its distribution leads to a sufficient condition for recurrence. 

Define 
\[M^{d,p} \coloneqq \sup_{\lambda \geq 0} \E[e^{\lambda - p_d^* - \hat p (1+U)\lambda}].\]
It follows from \cite[Proposition 2.6]{bailey2023critical} that 
\begin{align}
    \text{$M^{d,p}<1$ implies that $\SFM(d,p,\Poi(1))$ is recurrent. } \label{eq:suff}
\end{align}We define the critical value
\[q_d \coloneqq \inf\{p \colon M^{d,p} < 1\}\]
that corresponds to the threshold at which the proof technique in \cite{bailey2023critical, hoffman2016transience, johnson2016critical} fails. We do not have a proof, but find it plausible that $q_d$ is equal to the more natural critical value $\inf\{ p \colon \SFM(d,p,\Poi(1)) \text{ is recurrent}\}$. Our final result is that $q_d$ is strictly monotone.

\begin{theorem}\thlabel{thm:m2}
    $q_{d+1} < q_d$ for all $d \geq 2$.
\end{theorem}

This result is interesting because it supports the strict monotonicity claimed in Conjecture (ii). Moreover, it suggests that proving Conjecture (ii) may be difficult. The advantage of $q_d$ is that it is defined in terms of a concrete inequality satisfied by $M^{d,p}$. The self-similar nature of $\SFM(d,p,\Poi(1))$ as well as computational advantages unique to a Poisson-distributed number of frogs let us reduce this to analyzing functions. Even with these advantages the argument is not straightforward. The general case is unlikely to reduce to such tractable analysis.


\subsection{Organization}
We begin by proving \thref{thm:m1} in Section~\ref{sec:m1}. The result is needed to deduce \thref{thm:Sm}. In the next two sections we prove \thref{thm:Sm} and \thref{thm:m2}, respectively. The appendix contains additional details on the code we used to obtain \thref{thm:Sm} and Figure~\ref{fig:up-bounds}.




\section{Proof of \texorpdfstring{\thref{thm:m1}}{Theorem 2}} \label{sec:m1}

\begin{proof}
    Let $\SFM_n(d,p,\nu)$ denote the self-similar frog model with sleeping frogs placed at vertices within distance $n$ of the root. We will couple $\SFM_n(d,p,\nu)$ with a subprocess $\SFM_n'(d+1,p,\nu)$, dominated by $\SFM_n(d+1,p,\nu)$, that sometimes kills additional frogs. For $s \geq 0$, consider a time-changed version of the self-similar frog model that moves a randomly sampled awake frog at each step $s$. If the frog moves towards the root, then it takes one step and a new awake frog is sampled (with replacement) for step $s+1$. If the frog moves away from the root, then it samples a non-backtracking path to distance $n+1$ from the root and is killed there. The step $s$ is increased and a new frog, that has not been killed, is sampled. Let $\mathcal T^s_n(d,p,\nu)$ be the random subtree of sites within distance $n$ of the root that have been visited in $\SFM_n(d,p,\nu)$ after $s$ such steps.
    
    We take as our inductive hypothesis that there is an embedding $\psi^s \colon \mathcal T_n^s(d,p,\nu) \to \mathbb T_{d+1}$ that maps root to root ($\psi^s(\varnothing_d) = \varnothing_{d+1}$) and for each vertex $v \in \mathcal T_n^s(d,p,\nu)$
        \begin{enumerate}[label=(\roman*)]
            \item $\psi^s(v)$ has been visited in $\SFM'_n(d+1,p,\nu)$.
            \item There is a bijection between the frogs moving toward and away from the root at $v$ and $\psi^s(v)$.
        \end{enumerate}
  Once established, the inductive claim implies \thref{thm:m1}. This is because for any fixed $n$ the algorithm will 
  terminate with all frogs sleeping or killed after finitely many steps $s$. Thus the terminated algorithm produces the total number of visits to the root in $\SFM_n(d,p,\nu)$ which we have coupled to be identical as the process $\SFM_n'(d+1,p,\nu)$ that kills frogs and thus produces fewer root visits than $\SFM_n(d+1,p,\nu)$. Using monotonicity for the number of root visits in $n$ and taking $n \to \infty$ gives the desired result $V_{\SFM(d,p,\nu)} \preceq V_{\SFM(d+1,p,\nu)}$. 
        
         We now prove the inductive claim. Clearly (i) and (ii) are satisfied at $s=0$. It suffices to assume (i) and (ii) hold after $s$ steps and prove that they continue to hold after moving any one of the active frogs. Suppose that a frog $f$ at $v$ is selected to move on $\mathbb T_d$. Let $v' = \psi^s(v)$ and $f'$ be the corresponding frog at $v'$. Since any frogs that move away from the root are allowed to jump until being killed at distance $n+1$ from the root, we need only consider the cases that the previous jump of $f$ was towards the root or that $f$ was just awoken.


    Suppose that the last step $f$ took was towards the root. Both $f$ and its corresponding $f'$ at $v'$  will take another step towards the root with probability $\hat p$. We may then couple them to move in the same direction. If they move towards the root, then (i) and (ii) are preserved. Suppose $f$ and $f'$  move away from the root and that there are $n$ child vertices of $v$ that are yet to be visited. As $\psi^s$ is an embedding, there are $n+1$ child vertices of $v'$ that are yet to be visited. If $f$ does not move to a new site, then kill $f$ and $f'$. The frog $f$ moves to a new site with probability $n/d$. The frog $f'$ moves to a new site with  probability $(n+1)/(d+1) \geq n/d$. Thus, whenever $f$ moves to a new site $u$, there is a coupling that preserves the random walk law for $f'$ and has it move to a new site $u'$. We set $\psi^{s+1}(u) = u'$ and otherwise $\psi^{s+1} = \psi^s$. Further, we couple the number of particles discovered at $u$ and $u'$ to be the same. We then have $f$ and $f'$ repeat this coupling process until they reach distance $n+1$. 

    Suppose that the frogs $f$ and $f'$ were just woken at $v$ and $v'$, respectively. Suppose that $v$ has $n$ child vertices that are yet to be visited. Necessarily $v'$ has $n+1$ such vertices. Moreover, our requirement that frogs which have moved away from the root continue doing so until reaching distance $n+1$ from the root ensures that at least one child vertex below $v$ and $v'$ has been visited so that $n < d$. $f$ will move towards $\varnothing$ with probability $p_d^*< p_{d+1}^*$. So, if $f$ moves towards the root, we may couple $f'$ to do the same preserving (i) and (ii). $f$ will move away from the root to a new vertex with probability $(1- p_d^*)\f{n}{d}$. It is easy to verify that, so long as $n<d$, this is strictly less than the probability $(1-p_{d+1}^*)\f{n+1}{d+1}$ that $f'$ moves away to a new vertex. Thus, we may couple $f'$ to visit a new vertex whenever $f$ does. We then couple the number of frogs activated at the two sites. Lastly, if $f$ moves away from $\varnothing$ to an already visited vertex, then we kill both $f$ and $f'$. These rules were defined so that both $f$ and $f'$ move with the appropriate random walks, but $f'$ only discovers a new vertex when $f$ does. This preserves the embedding as well as (i) and (ii).     
\end{proof}

\section{Proof of \texorpdfstring{\thref{thm:Sm}}{Theorem 1}}\label{sec:pf-Sm}




First, we describe an efficient way to inductively compute the distribution of $U$ as $d$ is increased.

\begin{lemma}\thlabel{lem:inductive} 
   Set $\Phi_d = \exp\bigl(-\frac{1 - p_d^*}{d}\bigr)$ and $\Lambda_d = \exp\bigl(-\frac{(1-\hat{p})}{d-1}\lambda \bigr)$. Let $s_{1,0}(x,y) = 1$. For $d\geq 1$ and $0 \leq u \leq d-1$, define the sequence of functions $s_{d,u}(x,y)$ recursively by
\begin{align}
    s_{d,u}(x,y) & = \textstyle \binom{d-1}{u} \bigl(x y^{u+1}\bigr)^{d-1-u} s_{u+1,u}(x,y)  \text{\quad  for }0\leq u\leq d-2\text{, } \label{eq:rec1}  \\
    s_{d,d-1}(x,y) & = 1 - \textstyle \sum_{i = 0}^{d-2} s_{d,i}(x,y). \label{eq:rec2}
\end{align}
It holds that 
\begin{equation}
s_{d,u}(\Phi_d,\Lambda_d) = \P (U(d,p,\lambda) = u). \label{eq:s}
\end{equation}
\end{lemma}

\begin{proof}
We will prove \eqref{eq:s} inductively in $d$ with the following hypothesis:
\begin{equation}
s_{d,u}(\Phi_d,\Lambda_d) = \P(U (d,p,\lambda) =u ) \qquad \forall\, 0 \leq u \leq d-1 \label{eq:ih}.
\end{equation}
The base case $d=1$ is trivial since $s_{1,0}\equiv 1$. Now suppose that \eqref{eq:ih} holds for $s_{d-1,u}$ for all $0\leq u \leq d-2$.
To prove equality for $s_{d,u}$ for $0\leq u <d-1$, we define two events: 
    \begin{itemize}
        \item $\SA$ is the event that all of $\varnothing',v_1,v_2,\ldots, v_{u+1}$ are visited using only particles started at $\varnothing',v_1,v_2,\ldots, v_{u+1}$,
        \item $\NV$ is the event that particles from $\varnothing',v_1,v_2,\ldots,v_{u+1}$ do not visit any of the vertices $v_{u+2}, \ldots, v_d$.
    \end{itemize}
    Notice that for $0 \leq u\leq d-2$ 
    \[
    \P(U(d,p,\lambda) = u) = \binom{d-1}{u}\Pr(\NV\mid \SA)\Pr(\SA).
    \]
    The choice of $v_2,\ldots,v_{u+1}$ can be replaced by any $u$ vertices, leading to the $\binom{d-1}{u}$ factor at the front. The event $\SA$ can be viewed as all of the leaves being visited on the star with $u+1$ leaves but with transition probabilities corresponding to $\Phi_d$ and $\Lambda_d$. By our inductive hypothesis we then have $\Pr(\SA) = s_{u+1,u}(\Phi_d,\Lambda_d)$. 
    
    It remains to describe $\Pr(\NV \mid \SA)$. Conditional on the occurrence of \SA, there are $d-1-u$ \emph{outside} vertices that must not be visited. Using Poisson thinning, $\varnothing'$ sends an independent $\Poi\bigl(\frac{1-p_d^*}{d}\bigr)$ to each outside vertex. The probability all of these quantities are zero is $\Phi_d^{d-1-u}$. Again by Poisson thinning, each of the $u+1$ activated vertices ($v_1$ starts out activated) sends an independent $\Poi\bigl(\frac{1-\hat p}{d-1}\lambda\bigr)$ number of particles to each outside vertex. The probability this is zero is $\Lambda_d^{d-1-u}$. As there are $u+1$ activated vertices, all of these are zero with probability $\Lambda^{(d-1-u)(u+1)}_d$. Multiplying this with $\Phi_d^{d-1-u}$ gives the probability of $\NV$ given $\SA$. Thus, we have for all $0\leq u \leq d-2$ that 
    \[
        \P(U(d,p,\lambda) = u) = \binom{d-1}{u}\bigl(\Phi_d \Lambda_d^{u+1}\bigr)^{d-1-u} s_{u+1,u}(\Phi_d,\Lambda_d),
    \]
    which matches the $s_{d,u}(\Phi_d,\Lambda_d)$ in our recursive definition.
    
    Lastly, equality for $u=d-1$ is the observation that $\P(U(d,p,\lambda) = d-1)$ is complementary to $\P(U(d,p,\lambda) < d-1)$, which have the claimed formula by our previous reasoning. 
\end{proof}

\begin{proof}[Proof of \texorpdfstring{\thref{thm:Sm}}{}]
    Let $f(\lambda)=f^{d,p}(\lambda) = \E[e^{\lambda -p_d^* - \hat p (1+U) \lambda}]$. Suppose we prove that
    \[M^{d,p} \coloneqq \sup_{\lambda \geq 0} f^{d,p}(\lambda) <1.\] Then, \eqref{eq:suff} implies that $\SFM(d,p,\Poi(1))$ is recurrent. The stochastic comparison result in \cite{johnson2018stochastic} implies that $\SFM(d,p,1)$ is recurrent. The monotonicity result in \thref{thm:m1} further ensures that $\SFM(d',p,1)$ is recurrent for all $d'\geq d$. We then apply the observation that $V_{\SFM(d,p,1)} \preceq V_{\FM(d,p,1)}$ in \cite{guo2022minimal} to conclude that $\FM(d,p,1)$ is recurrent and thus $S_d\leq p$.
    
\thref{lem:inductive} allows us to efficiently compute the distribution of $U$ as we increase $d$. This formula only involves powers of exponential functions of $\lambda$. The method devised in \cite{bailey2023critical} gives a way to rigorously prove that $f$ has a single maximum on $[0,\infty)$. To accomplish this we make the change of variables $f(-c\log(y)) = g(y)$. If $c$ is an appropriately chosen rational number, then $g$ is a polynomial. The \texttt{CountRoots} function in Mathematica helps us to confirm with Sturm's theorem that $g'$ has a single root in $(0,1]$. So, $g$ (and thus $f$) has a unique global maximum which we prove is strictly less than $1$. We implemented this approach for $2 \leq d \leq 13$. For larger $d$ we experience runtime issues with this computer-assisted rigorous proof. More details can be found in the appendix and the code documentation at \href{https://github.com/fredcheng02/frog-model}{https://github.com/fredcheng02/frog-model}.
\end{proof}

\section{Proof of \texorpdfstring{\thref{thm:m2}}{Theorem 3}}

We first prove that $U$ is monotone in $d$. A similar, but less general observation was made in \cite{bailey2023critical}. 

\begin{lemma}\thlabel{lem:Um}
$U(d,p,\lambda) \preceq U(d+1,p,\lambda)$.
\end{lemma}

\begin{proof}
We start with a general observation about coupon collectors collecting from different sized sets of coupons. Namely, whoever has a larger set of coupons collects at least as many coupons as the collector with less options. To be more precise, suppose that two coupon collectors are independently collecting coupons. The first collector collects from a set of $n$ coupons and the second collector collects from a set of $n+1$ coupons. Let $C_t^i$ be the number of distinct coupons collected by collector $i=1,2$ after each takes $t$ uniform draws from their set of coupons.

\begin{claim}\thlabel{claim:coupon} There is a coupling so that $C_t^1 \leq C_t^2$ for all $t \geq 0$. 
\end{claim}

\begin{proof}
The inequality holds at $t=0$. Suppose it is true after collecting $t$ coupons. If $C_t^1 < C_t^2$, then regardless of the $(t+1)$th draw we will have $C_{t+1}^1 \leq C_{t+1}^2$. Suppose that $C_t^1 = C_t^2= c \leq n$. For the $(t+1)$th draw, the first collector has probability $(n-c)/n$
of collecting a new coupon. The second collector has probability $(n+1-c)/(n+1).$
It is easy to check that the second probability is at least as large as the first. Thus, we may couple the two processes so that the second collector discovers a new coupon whenever the first collector does. This gives $C_{t+1}^1 \leq C_{t+1}^2$. 
\end{proof}

Now we will describe $U(d,p,\lambda)$ and $U(d+1,p,\lambda)$ in terms of coupon collector processes. Collector A is collecting coupons uniformly from coupons $\{1,\hdots, d\}$. Collector B is collecting coupons uniformly from coupons $\{1,\hdots, d+1\}$. Suppose further that each collector has initially collected coupon 1. 

At the first step Collector A samples $\Poi\bigl((1-p_d^*)\f{d-1}{d}\bigr)$  coupons uniformly from $\{2,\hdots, d\}$, and Collector B samples $\Poi\bigl((1-p_{d+1}^*)\f{d}{d+1}\bigr)$ coupons uniformly from $\{2,\hdots, d+1\}$. A specific instance of the inequality in the last paragraph of Section~\ref{sec:m1} is that for $p<1/2$ and $d \geq 2$, 
$$(1- p_d^*) \f{d-1}{d} < (1-p_{d+1}^*) \f{d}{d+1}.$$
Thus, we may couple Collector B to collect at least as many total coupons from $\{2,\hdots, d+1\}$ as Collector A does from $\{2,\hdots, d\}$. By \thref{claim:coupon} there is a coupling so that Collector B ends up with at least as many distinct coupons as Collector A. 

For each distinct coupon collected, say coupon $i$, the collectors receive an independent $\Poi((1-\hat p) \lambda)$ distributed number of additional uniform draws from the set of coupons minus coupon $i$. We may couple the number of uniform draws from each newly discovered coupon and repeatedly apply \thref{claim:coupon} to ensure that Collector B always has at least as many distinct coupons as Collector A. In particular, this is true once no new coupons are discovered and the collecting ends. Comparing to the definition of $U$, the number of distinct coupons collected by Collectors A and B are distributed as $U(d,p,\lambda)$ and $U(d+1,p,\lambda)$, respectively. This gives the claimed stochastic dominance. 
\end{proof}

Recall that $f^{d,p}(\lambda) = \E[e^{\lambda -p_d^* - \hat p (1+U) \lambda}]$ and $M^{d,p} = \sup_{\lambda \geq 0} f^{d,p}(\lambda)$.
\begin{lemma}\thlabel{lem:M} For $p \in (\frac{1}{d+1},\frac{1}{2})$
    \begin{enumerate}[label = (\roman*)]
        \item $M^{d,p}$ is continuous in $p$.
        \item $M^{d+1,p}<M^{d,p}$.
    \end{enumerate}
\end{lemma}

\begin{proof}[Proof of (i)]
For now fix $p$ and $d$. Let $g(y) = f(-\log(y))$. Recall \[\P(U(d,p,\lambda) = u) = s_{d,u}(\Phi_d,\Lambda_d)\] from \thref{lem:inductive}. Making this replacement, we have that $f\colon [0,\infty) \to \R$ given by \[f(\lambda) = e^{-p_{d}^{*}} \sum_{u=0}^{d-1}e^{(1 -\hat{p}(1+u) )\lambda}  s_{d,u}(\Phi_d,\Lambda_d)\] becomes $g\colon(0,1] \to \R$ with \[
    g(y) = e^{-p_{d}^{*}}\sum_{u =0}^{d-1}y^{\hat{p}(1+u) - 1}s_{d,u}\bigl(\Phi_d,y^{\frac{1-\hat{p}}{d-1}}\bigr)
\]
using $\Lambda_d = \exp\bigl(-\frac{(1-\hat{p})}{d-1}\lambda \bigr)$. Now $M^{d,p} = \sup_{y \in (0,1]}g^{d,p}(y)$. 

It is easy to show via induction that the $y$ terms in $s_{d,u}(\Phi_d,y^{\frac{1-\hat{p}}{d-1}})$ have nonnegative exponents. So, when $0 \leq u \leq d-2$, the exponent of $y$ in the summand is 
\[
(1-\hat{p})\frac{(d-1-u)(u+1)}{(d-1)}+\bigl[\text{exponent of }y\text{ in }s_{u+1,u}\bigl(\Phi_d,y^{\frac{1-\hat{p}}{d-1}}\bigr)\bigr] + [\hat{p}(1+u)-1].
\]
The above expression is nonnegative since all of its terms are nonnegative, and equals zero if and only if $u = 0$. 
The summand when $u=0$ becomes $\Phi^{d-1}_d=\exp\bigl(-\frac{d-1}{d}(1-p_d^*)\bigr)$. If $u=d-1$, then our assumption that $p > \f{1}{d+1}$ ensures that $\hat{p} > 1/d$. So the $y$-exponents in the $u = d-1$ summand are strictly positive.



$g(y)$ is now a sum of terms, each with nonnegative $y$-exponents. Under the customary definition $0^0 = 1$, we see $g(y)$ is continuous on the interval $[0,1]$.

Now define the function $G\colon (\frac{1}{d+1},\frac{1}{2})\times[0,1] \to \mathbb R$ such that $G(p,y)=g^{d,p}(y)$ for every fixed $d$. It is clear that $G$ is continuous on its domain (since $\frac{\partial G}{\partial p}$ and $\frac{\partial G}{\partial y}$ are both continuous in some open set containing the domain). Consider any $p_{0}\in(\frac{1}{d+1},\frac{1}{2})$. Take any $r_1 < p_0$ and $r_2 > p_0$ so that $p_{0}\in[r_{1},r_{2}]\subseteq(\frac{1}{d+1},\frac{1}{2})$, then $G$ is uniformly continuous on the compact set $[r_{1},r_{2}]\times[0,1]$. In particular we have for all $\gamma>0$, there exists some $0<\delta\leq\min\{p_0-r_1,r_2-p_0\}$ such that for any $y\in(0,1]$, if $\abs{p-p_{0}}<\delta$, then $\abs{G(p,y)-G(p_{0},y)}<\gamma$.

To show $M^{d,p}=\sup_{y\in(0,1]}G(p,y)$ is continuous in $p$, note
\[
\Bigl\lvert\sup_{y\in(0,1]}G(p,y)-\sup_{y\in(0,1]}G(p_{0},y)\Bigr\rvert \leq\sup_{y\in(0,1]}\abs{G(p,y)-G(p_{0},y)}.
\]
By the uniform continuity argument above we know immediately that $M^{d,p}$ is continuous at all $p\in(\frac{1}{d+1},\frac{1}{2})$.
\end{proof}

\begin{proof}[Proof of (ii)]
In (i) we said $g(y)$ is continuous on the compact interval $[0,1]$, which implies $g(y)$ attains its maximum on $[0,1]$. Note that 
\[
g(0)=
\exp\Bigl(-p_{d}^{*}-\frac{d-1}{d}(1-p_d^*)\Bigr)=\exp\Bigl(-\frac{1}{d}p_{d}^{*}-\frac{d-1}{d}\Bigr)
\]
and $g(1) = e^{-p_d^*}$. By $p < 1/2$ we have $p_d^* < 1$, which implies $-\frac{1}{d}p_{d}^{*}-\frac{d-1}{d}<-p_{d}^{*}$. It follows that $g(1) > g(0)$, and hence $g(y)$ attains its maximum on $(0,1]$, i.e., \[M^{d,p} = \max_{y \in  (0,1]} g^{d,p}(y).\]

We can write $f^{d,p}(\lambda)= e^{-p_{d}^{*} + \lambda}\E[e^{-\hat{p}(1+U) \lambda}].$ By \thref{lem:Um} we have $U(d,p,\lambda)\preceq U(d+1,p,\lambda)$. 
Since $p_{d+1}^{*}>p_{d}^{*}$, it follows that  $f^{d+1,p}(\lambda) < f^{d,p}(\lambda).$
%
Therefore $g^{d+1,p}(y)<g^{d,p}(y)$ for all $y\in(0,1]$.
Since $g$ attains its maximum on $(0,1]$, $M^{d+1,p}<M^{d,p}$, as desired.
\end{proof}





It is now quick to deduce our final result.

\begin{proof}[Proof of \thref{thm:m2}]
    Referring to the statements in \thref{lem:M}. (i) implies that $M^{d,q_d} \leq 1$. (ii) implies that $M^{d+1,q_d} <1$. Another application of (i) implies that $M^{d+1,q_d-\epsilon} <1$ for some $\epsilon >0$. Hence $q_{d+1} \leq q_d -\epsilon$.
\end{proof}

\section{Appendix}
We explain how the exact bounds in \thref{thm:Sm} and the approximate bounds in Figure~\ref{fig:up-bounds} are obtained. The code and further explanation can be found in the Jupyter Notebooks at \href{https://github.com/fredcheng02/frog-model}{https://github.com/fredcheng02/frog-model}. The code is written in SageMath, but we imported the \texttt{CountRoots} function from Mathematica at some moment, as previously noted.

We mentioned in the proof of \thref{thm:Sm} that we make a change of variable $f(-c\log(y)) = g(y)$. The approach in \cite{bailey2023critical} was to visually inspect the formula to determine which $c$ should be chosen so that $g(y)$ gives us a polynomial (integer exponents). This is possible for $d \leq 4$, but problematic for larger $d$ when the expressions involve an exponentially increasing number of terms. Our approach is to take $p$ to be an irreducible fraction $\frac{a}{b} \in (\frac{1}{d+1},\frac{1}{2})$ and set $c = (b-a)(d-1)$. Then \[
    g(y) = \exp\Bigl(\frac{a-ad}{bd-ad-d}\Bigr)\sum_{u=0}^{d-1}y^{(d-1)[(u+2)a-b]}s_{d,u}(\Phi_d,y^{b-2a}),
\]
where $\Phi_d = \exp\bigl(\frac{2a-b}{bd-ad-a}\bigr)$. Using the same argument as in \thref{lem:M} (i), one can show that this change of variables always makes $g(y)$ a polynomial. We then have the power of exact algorithms to show that $g$ has a unique maximum on $(0,1]$, which we prove is strictly less than $1$. The bounds obtained in \thref{thm:Sm} are nearly the best possible using this method. We use continued fraction approximation to find the $p$'s such that $0.9994 < \sup_{y \in (0,1]}g^{d,p}(y) < 1$ for all $2\leq d \leq 13$.

In finding the maximum of $g$ on $(0,1]$ above, we need to use the \texttt{find\_root} function in SageMath to obtain the root of $g'$. The Brent's method used in \texttt{find\_root} only allows fixed machine precision, and fails to give a correct root when $d \geq 14$, given the fast-growing complexity of the exact symbolic expression for $g(y)$. This led us to numerically approximate the bounds for higher $d$'s instead.

To obtain the values in Figure \ref{fig:up-bounds}, we work with the original $f(\lambda)$. It appears that $\argmax_{\lambda \geq 0} f^{d,p}(\lambda)$ converges monotonically down to some value in $[0,1)$ as $d\to \infty$, and $f$ has small curvature around its peak. For a chosen $p$, if we can verify that $f^{d,p}(0),f^{d,p}(0.01),\ldots,f^{d,p}(0.99)$ are all 
 strictly less than $1$, then we check for a slightly smaller $p$ (e.g., we decrease $p$ by 0.0001 in our code) if they are still less than 1. If not, then we have just found the approximate $p$ such that $\sup_{\lambda \in [0,\infty)}f^{d,p}(\lambda) \approx 1$.

Our code also lets us verify if a specific $p$ can work as an upper bound to $S_m$. For large $m$ we have to adjust the real number precision accordingly to avoid numerical errors. For example, we have checked that the upper bounds fall below $0.18$ for the first time at around $m = 230$, which again supports our conjecture that $p_d$ converges at rate $d^{-1/2}$. 

\bibliographystyle{alpha}

\bibliography{frog}

\newcommand{\etalchar}[1]{$^{#1}$}
\begin{thebibliography}{DGH{\etalchar{+}}18}

\bibitem[AMP02]{alves2002shape}
Oswaldo~SM Alves, Fabio~P Machado, and S~Yu Popov.
\newblock The shape theorem for the frog model.
\newblock {\em The Annals of Applied Probability}, 12(2):533--546, 2002.

\bibitem[BDD{\etalchar{+}}18]{beckman2018asymptotic}
Erin Beckman, Emily Dinan, Rick Durrett, Ran Huo, and Matthew Junge.
\newblock Asymptotic behavior of the {B}rownian frog model.
\newblock {\em Electronic Journal of Probability}, 23:1--19, 2018.

\bibitem[BFJ{\etalchar{+}}19]{beckman2019frog}
Erin Beckman, Natalie Frank, Yufeng Jiang, Matthew Junge, and Si~Tang.
\newblock The frog model on trees with drift.
\newblock {\em Electronic Communications in Probability}, 24:1--10, 2019.

\bibitem[BJL23]{bailey2023critical}
Emma Bailey, Matthew Junge, and Jiaqi Liu.
\newblock Critical drift estimates for the frog model on trees.
\newblock {\em arXiv:2303.15517}, 2023.

\bibitem[DGH{\etalchar{+}}18]{dobler2018reccurence}
Christian D{\"o}bler, Nina Gantert, Thomas H{\"o}felsauer, Serguei Popov, and
  Felizitas Weidner.
\newblock {Recurrence and transience of frogs with drift on $\mathbb{Z} ^d$}.
\newblock {\em Electronic Journal of Probability}, 23:1 -- 23, 2018.

\bibitem[GS09]{gantert2009recurrence}
Nina Gantert and Philipp Schmidt.
\newblock Recurrence for the frog model with drift on $\mathbb{Z}$.
\newblock {\em Markov Process. Related Fields}, 15(1):51--58, 2009.

\bibitem[GTW22]{guo2022minimal}
Chengkun Guo, Si~Tang, and Ningxi Wei.
\newblock On the minimal drift for recurrence in the frog model on $d$-ary
  trees.
\newblock {\em The Annals of Applied Probability}, 32(4):3004--3026, 2022.

\bibitem[HJJ16]{hoffman2016transience}
Christopher Hoffman, Tobias Johnson, and Matthew Junge.
\newblock From transience to recurrence with {P}oisson tree frogs.
\newblock {\em The Annals of Applied Probability}, 26(3):1620--1635, 2016.

\bibitem[HJJ17]{hoffman2017recurrence}
Christopher Hoffman, Tobias Johnson, and Matthew Junge.
\newblock Recurrence and transience for the frog model on trees.
\newblock {\em The Annals of Probability}, 45(5):2826--2854, 2017.

\bibitem[HJJ19]{hoffman2019infection}
Christopher Hoffman, Tobias Johnson, and Matthew Junge.
\newblock Infection spread for the frog model on trees.
\newblock {\em Electronic Journal of Probability}, 24:1--29, 2019.

\bibitem[JJ16]{johnson2016critical}
Tobias Johnson and Matthew Junge.
\newblock The critical density for the frog model is the degree of the tree.
\newblock {\em Electronic Communications in Probability}, 21:1--12, 2016.

\bibitem[JJ18]{johnson2018stochastic}
Tobias Johnson and Matthew Junge.
\newblock Stochastic orders and the frog model.
\newblock {\em Annales de l'Institut Henri Poincar{\'e}, Probabilit{\'e}s et
  Statistiques}, 54(2):1013--1030, 2018.

\bibitem[JR19]{johnson2019sensitivity}
Tobias Johnson and Leonardo~T Rolla.
\newblock Sensitivity of the frog model to initial conditions.
\newblock {\em Electronic Communications in Probability}, 24:1--9, 2019.

\bibitem[Pop01]{popov2001frogs}
S~Yu Popov.
\newblock Frogs in random environment.
\newblock {\em Journal of Statistical Physics}, 102(1):191--201, 2001.

\bibitem[Pop03]{popov2003frogs}
Serguei~Yu Popov.
\newblock Frogs and some other interacting random walks models.
\newblock 2003.

\bibitem[RS04]{ramirez2004asymptotic}
Alejandro~F Ram{\'\i}rez and Vladas Sidoravicius.
\newblock Asymptotic behavior of a stochastic combustion growth process.
\newblock {\em Journal of the European Mathematical Society}, 6(3):293--334,
  2004.

\bibitem[TW99]{telcs1999branching}
Andr{\'a}s Telcs and Nicholas Wormald.
\newblock Branching and tree indexed random walks on fractals.
\newblock {\em Journal of applied probability}, 36(4):999--1011, 1999.

\end{thebibliography}

\end{document}